\documentclass[12pt,a4paper,english]{article}

\newcommand{\beq}{\begin{equation}}
\newcommand{\eeq}{\end{equation}}

\usepackage[latin1]{inputenc}
\usepackage{babel}
\usepackage[centertags]{amsmath}
\usepackage{amsfonts}
\usepackage{amssymb}
\usepackage{color}
\usepackage{amsthm}
\usepackage{newlfont}
\usepackage{graphicx}
\usepackage{dsfont}
\usepackage{geometry}
\usepackage{mathrsfs}
\usepackage{authblk}
\usepackage{verbatim}
\usepackage{hyperref}

\newtheorem{theorem}{Theorem}[section]
\newtheorem{lemma}[theorem]{Lemma}
\newtheorem{coroll}[theorem]{Corollary}
\newtheorem{prop}[theorem]{Proposition}

\newtheorem{remark}[theorem]{Remark}

\newcommand{\msc}[1]{\textbf{MSC2010:} #1.}
\newcommand{\keywords}[1]{\textbf{Key words:} #1.}
\newcommand{\ackn}[1]{\textbf{Acknowledgments:} #1.}

\def\hs{\hspace}

\def\RR{\mathbb R}
\def\NN{\mathbb N}
\def\EE{\mathsf E}
\def\PP{\mathsf P}
\def\CC{\mathcal C}
\def\DD{\mathcal D}
\def\eps{\varepsilon}
\def\cF{{\cal F}}
\def\supp{\text{supp}}

\makeatletter  
\@addtoreset{equation}{section}

\makeatother  
\begin{document}
\title{\textbf{Integral equations for Rost's reversed barriers: existence and uniqueness results}
}
\author{Tiziano De Angelis\thanks{School of Mathematics, University of Leeds, Woodhouse Lane~LS2 9JT Leeds, UK; \texttt{t.deangelis@leeds.ac.uk}}\:\:\:\:and\:\:\:\:Yerkin Kitapbayev\thanks{Questrom School of Business, Boston University, 595 Commonwealth Avenue, 02215, Boston, MA, USA; \texttt{yerkin@bu.edu}}}
\maketitle

\begin{abstract}
We establish that the boundaries of the so-called Rost's reversed barrier are the unique couple of left-continuous monotonic functions solving a suitable system of nonlinear integral equations of Volterra type. Our result holds for atom-less target distributions $\mu$ of the related Skorokhod embedding problem.

The integral equations we obtain here generalise the ones often arising in optimal stopping literature and our proof of the uniqueness of the solution goes beyond the existing results in the field.
\end{abstract}
\msc{60G40, 60J65, 60J55, 35R35, 45D05}
\vspace{+8pt}

\noindent\keywords{Skorokhod embedding, Rost's reversed barriers, optimal stopping, free-boundary problems, Volterra integral equations}
\section{Introduction}

The celebrated Skorokhod embedding problem formulated in the 1960's by Skorokhod \cite{Skor65} can be stated as follows: given a probability law $\mu$, find a stopping time $\tau$ of a standard Brownian motion $W$ such that $W_\tau$ is distributed according to $\mu$. {Several solutions to this problem have been found over the past fifty years in a variety of settings (see for instance \cite{AY79}, \cite{Per86}, \cite{Rost71}, \cite{Root}, \cite{Val83}, \cite{CH06} and \cite{COT15} among many others; see also the survey \cite{Ob04}) and very often they find important applications ranging from general martingale theory (see for instance \cite{BrHoRo01} and \cite{Ro93}) to mathematical finance (see \cite{Dup05} or \cite{Hob98} amongst many others, and \cite{Hob11} for a survey).} 

Here we consider a solution proposed by Rost \cite{Rost71} which can be given in terms of first hitting times of the time-space process $(t,W_t)_{t\ge0}$ to a set, depending on $\mu$, usually called \emph{reversed barrier} \cite{Chacon85}. The boundaries of this set can be expressed as two monotonic functions of time (see \cite{Cox-Pe13}) and we denote them $t\mapsto s_\pm(t)$ (here we understand $s_\pm(t)=s_\pm(t;\mu)$ depending on $\mu$ but prefer to adopt a simpler notation as no confusion should arise).

Given a atom-less target distribution $\mu$ (possibly singular like the Cantor distribution) we use purely probabilistic methods to characterise the couple $(s_+,s_-)$ as the unique solution of suitable nonlinear integral equations of Volterra type. To achieve that we use a known connection between Rost's solution of the embedding problem and an optimal stopping problem (see for instance \cite{Cox-Wang13b}, \cite{McCon91} and \cite{DeA15-SK}) so that our contribution becomes twofold: on the one hand we provide a new way of characterising and evaluating $s_\pm$ and on the other hand we extend in several directions the existing literature concerning uniqueness of solutions to certain integral equations.

To the best of our knowledge until now there exist essentially two ways of computing Rost's reversed barriers: a constructive way and a PDE way. The former approach was proposed by Cox and Peskir \cite{Cox-Pe13} who approximate the boundaries $s_\pm$ via a sequence of piecewise constant boundaries (in our notation $(s^n_\pm)_{n\ge0}$). They prove suitable convergence of $s^n_\pm$ to $s_\pm$ as $n\to\infty$ and for each $n$ they also provide equations that can be solved numerically to find $s^n_\pm$ (however no equations could be provided for $s_\pm$ in the limit). In practice for $n$ large enough one can compute an accurate approximation of $s_\pm$. The PDE approach instead relies upon the aforementioned connection to optimal stopping as proved in \cite{Cox-Wang13b} through viscosity solution of variational inequalities. The idea is that one should numerically solve a variational inequality to obtain the value function of the stopping problem and then use it to compute the stopping set. The boundaries of such stopping set indeed coincide with $s_\pm$ up to a suitable time reversal. An approach of this kind was used in Section 4 of \cite{GODR14}, although in the different context of Root's solutions to Skorokhod embedding.

Here we do use the connection to optimal stopping but instead of employing PDE theory we only rely upon stochastic calculus. Our work offers a point of view different to the one of \cite{Cox-Pe13} and \cite{Cox-Wang13b} since we directly characterise the boundaries $s_\pm$ of Rost's barriers as the unique solutions to specific equations. We do not develop a constructive approach nor we use an indirect approach via variational inequalities. The computational effort required to evaluate our equations numerically is rather modest and hinges on a simple algorithm which is widely used in the optimal stopping literature (see Section \ref{sec:num}). It is worth mentioning that McConnell \cite{McCon91} provided a family of uncountably many integral equations for $s_\pm$ which appear to be very different from ours as they involve at the same time the boundaries and (in our language) the value function of the optimal stopping problem. Solvability of these equations seems to be left as an open question and we are not aware of any further contributions in the direction of finding its answer.

To clarify our contribution to the literature on integral equations and optimal stopping we would like to recall that nonlinear equations of Volterra type arise frequently in problems of optimal stopping on finite time horizon. Such equations are indeed used to compute the related optimal stopping boundaries (see \cite{Pes-Shir} for a collection of examples) and their study in this context goes back to McKean \cite{McK65} and Van Moerbeke \cite{VM} among others (see also \cite{CJM92}, \cite{Jacka}, \cite{KIM} and \cite{Myn92} for the American put exercise boundary). Volterra equations are also widely studied in connection with the Stefan problem in PDE theory (see for example \cite{Can} and Chapter 8 in \cite{Fri64}).

The particular equations that we consider here (see Theorem \ref{thm:integreq} below) are more general than the ones usually addressed in optimal stopping literature and in PDE theory. All examples we know from optimal stopping involve integrals with respect to the Lebesgue measure on $\RR_+\times\RR$. Our equations instead involve integrals with respect to a product measure $dt\times(\nu-\mu)(dx)$ on $\RR_+\times\RR$, where $\mu$ and $\nu$ are rather general probability measures. Only when both $\mu$ and $\nu$ have density with respect to the Lebesgue measure we can rely upon the existing results.   

As for the integral equations used in PDE theory in the context of the Stefan problem, there is no direct way of employing 
the related existence and uniqueness results (see for example Chapter 8 in \cite{Can} and compare to our \eqref{eq:intEq}). This is due to the need for a regularity of the terms in the equation that cannot be achieved when considering Rost's solutions to the Skorokhod embedding problem.

In this sense our first contribution is to establish existence of solutions for such broader class of integral equations. Secondly, the question of uniqueness of the solution was a long standing problem for those equations normally addressed in optimal stopping theory. It found a positive answer in the paper by Peskir \cite{Pe-0} where uniqueness was proved in the class of continuous functions. Here we prove that such uniqueness holds for a more general class of integral equations and, more interestingly, we prove that the solution is indeed unique in the class of \emph{right-continuous} functions, thus extending Peskir's previous results (note that our result obviously holds for all the cases covered by \cite{Pe-0}). We remark that although the line of proof is based on a \emph{4-step} procedure as in \cite{Pe-0}, the greater generality of our setting requires new arguments to prove each one of the four steps so that our proof could not be derived by the original one in \cite{Pe-0}.

Before concluding this introduction we would like to point out that our work is {somehow} related to a recent one by Gassiat, Mijatovi\'c and Oberhauser \cite{GMO15}. There the authors consider another solution of the Skorokhod embedding problem, namely Root's solution \cite{Root}, which is also given in terms of a barrier set in time-space. The boundary of Root's barrier is expressed as a function of space $x\mapsto r(x)$ rather than a function of time. Under the assumption of atom-less target measures $\mu$ (as in our case) \cite{GMO15} prove that the boundary  $r(\,\cdot\,)$ solves a Volterra-type equation which is reminiscent of the ones obtained here. Uniqueness is provided via arguments based on viscosity theory in the special cases of laws $\mu$'s which produce continuous boundaries $r(\,\cdot\,)$. This additional requirement excludes for instance continuous distributions singular with respect to the Lebesgue measure (e.g.~Cantor distribution) which are instead covered in our paper.

It is important to notice that although Rost's and Root's solutions are conceptually related there seems to be no easy way of obtaining one from the other. Therefore even if equations found here and those in \cite{GMO15} may look similar we could not establish any rigorous link. Rost's barriers enjoy nicer regularity properties than Root's ones (e.g.~monotonicity and left continuity) and this allowed us to prove uniqueness of solution for the integral equations even in the case of continuous singular measures $\mu$. For the same reason the Volterra equations that we obtain here are relatively easy to compute numerically; indeed we can run our algorithm forward in time starting from known initial values $s_\pm(0)$ of the boundaries at time zero. The equations obtained in \cite{GMO15} are instead more complicated as there are no clear boundary conditions on $r(\,\cdot\,)$ that can be used in general to initialise the algorithm. In fact the numerical evaluation of the boundaries in \cite{GMO15} is restricted to the class of $\mu$'s that produce symmetric, continuous, monotonic on $[0,\infty]$, maps $x\mapsto r(x)$. Finally the existence results in our paper and in \cite{GMO15} are derived in completely different ways.

The rest of the paper is organised as follows. In Section \ref{sec:sett} we provide some background material on Rost's barriers, state our main result in Theorem \ref{thm:integreq} and recall the link to a suitable optimal stopping problem which is needed for the proof. Section \ref{sec:proof} is entirely devoted to the proof of Theorem \ref{thm:integreq}. Finally, in Section \ref{sec:num} we address the numerical solution of the integral equations.

\section{Setting, background material and results}\label{sec:sett}


Let $(\Omega,\cF,\PP)$ be a probability space, $B:=(B_t)_{t\ge0}$ a one dimensional standard Brownian motion 
and denote $(\cF_t)_{t\ge0}$ the natural filtration of $B$ augmented with $\PP$-null-sets. Throughout the paper we will equivalently use the notations $\EE f(B^x_t)$ and $\EE_x f(B_t)$, $f:\RR\to \RR$, Borel-measurable, to refer to expectations under the initial condition $B_0=x$. We also use the notation $X\sim\mu$ for a random variable $X$ with law $\mu$.

Let $\mu$ and $\nu$ be probability measures on $\RR$. We denote by $F_\mu(x):=\mu((-\infty,x])$ and $F_\nu(x):=\nu((-\infty,x])$ the (right-continuous) cumulative distributions functions of $\mu$ and $\nu$. Throughout the paper we will use the following notation:
\begin{align}
\label{not:a}&a_+:=\sup\{x\in\RR\,:\,x\in\supp\,\nu\}\quad\text{and}\quad a_-:=-\inf\{x\in\RR\,:\,x\in\supp\,\nu\}\\[+4pt]
\label{not:mu}&\mu_+:=\sup\{x\in\RR\,:\,x\in\supp\,\mu\}\quad\text{and}\quad\mu_-:=-\inf\{x\in\RR\,:\,x\in\supp\,\mu\}
\end{align}
and for the sake of simplicity but with no loss of generality we will assume $a_\pm\ge 0$. We also make the following assumptions which are standard in the context of Rost's solutions to the Skorokhod embedding problem (see for example \cite{Cox-Pe13} and in particular Remark 2 therein).
\begin{itemize}
\item[(D.1)] There exist numbers $\hat{b}_+\ge a_+$ and $\hat{b}_-\ge a_-$ such that $(-\hat{b}_-,\hat{b}_+)$ is the largest interval containing $(-a_-,a_+)$ with $\mu((-\hat{b}_-,\hat{b}_+))=0$;
\item[(D.2)] If $\hat b_+=a_+$ (resp.~$\hat b_-=a_-$) then $\mu(\{\hat{b}_+\})=0$ (resp.~$\mu(\{-\hat{b}_-\})=0$). 
\end{itemize}
It should be noticed in particular that in the canonical example of $\nu(dx)=\delta_0(x)dx$ the above conditions hold for any $\mu$ such that $\mu(\{0\})=0$. We also observe that assumption (D.2) is made in order to avoid solutions of the Skorokhod embedding problem involving randomised stopping times whereas assumption (D.1) guarantees that at each time $t>0$ the $t$-section of the Rost barrier is the union of at most two connected components.

Now we recall a particular result recently obtained by Cox and Peskir \cite{Cox-Pe13} in the context of a more general work on Rost's solutions of the Skorokhod embedding problem. Here we give a different statement for notational convenience and the next theorem summarises \cite[Thm.~1, Prop.~9 and Thm.~10]{Cox-Pe13}.
\begin{theorem}[Cox \& Peskir \cite{Cox-Pe13}]\label{thm:CoxPe}
Let $W^\nu:=(W^\nu_t)_{t\ge0}$ be a standard Brownian motion with initial distribution $\nu$ and let (D.1) and (D.2) hold. There exists a unique couple of left continuous, increasing functions $s_+:\RR_+\to\RR_+\cup\{+\infty\}$ and $s_-:\RR_+\to\RR_+\cup\{+\infty\}$ such that if we define
\begin{align}\label{eq:sigma*}
\sigma_*:=\inf\big\{t>0\,:\,W^\nu_t\notin \big(-s_-(t),s_+(t)\big)\big\}
\end{align}
then $W^\nu_{\sigma_*}\sim\mu$.

Moreover, letting $\Delta s_\pm(t):=s_\pm(t+)-s_\pm(t)\ge 0$, for any $t\ge 0$ such that $s_\pm(t)<+\infty$ it also holds
\begin{align}
\label{eq;jumpb}&\Delta s_+(t)>0\quad\Leftrightarrow\quad \mu\big(\big(s_+(t),s_+(t+)\big)\big)=0\\
\label{eq;jumpb2}&\Delta s_-(t)>0\quad\Leftrightarrow\quad \mu\big(\big(-s_-(t+),-s_-(t)\big)\big)=0.
\end{align}
\end{theorem}
\vspace{+5pt}

We point out that the above theorem is stated in \cite[Thm.~1]{Cox-Pe13} for the special case $\nu(dx)=\delta_0(x)dx$. Remark 2 in \cite{Cox-Pe13} clarifies that the result and its proof remain identical for any initial distribution satisfying (D.1) and (D.2). Alternatively the same result may be found in \cite[Thm.~2.3]{DeA15-SK} where a different proof relying upon optimal stopping theory is provided for general $\nu$. 

Since $s_\pm$ are the boundaries of the Rost barrier associated to $\mu$, it is clear that flat portions of $s_+$ or $s_-$ must correspond to atoms of $\mu$. Hence the next corollary follows.
\begin{coroll}\label{cor:flat}
If the measure $\mu$ is atom-less then the boundaries $s_+$ and $s_-$ are strictly monotone.
\end{coroll}
\vspace{+5pt}

In preparation for the statement of our main result let us denote
\begin{align}\label{def:transd}
p(t,x,s,y):=\frac{1}{\sqrt{2\pi(s-t)}}e^{-\frac{(x-y)^2}{2(s-t)}},\quad\text{for $t<s$, $x,y\in\RR$}
\end{align}
the Brownian motion transition density and let $(L^y_s)_{s\ge0}$ be the local time of $B$ at a point $y\in\RR$.

As in Section 4 of \cite{DeA15-SK} we introduce the (generalised) inverse of $s_\pm$ defined by
\begin{align}\label{eq:invbarrier00}
\varphi(x):=
\left\{
\begin{array}{ll}
\inf\{t\ge 0\,:\,-s_-(t)< x \}, & x\le -s_-(0),\\[+4pt]
0, & x\in(-s_-(0),s_+(0)),\\[+4pt]
\inf\{t\ge0\,:\,s_+(t)> x\}, & x\ge s_+(0).
\end{array}
\right.
\end{align}
Note that for $t>0$, one has $x\in(-s_-(t),s_+(t))$ if and only if $\varphi(x)<t$.
Moreover $\varphi$ is positive, decreasing left-continuous on $\RR_-$ and increasing right-continuous on $\RR_+$ (hence upper semi-continuous).

Now that we have introduced the necessary background material we can provide the main result of this work, i.e.~the characterisation of the boundaries $s_\pm$ as the unique solution of suitable nonlinear integral equations of Volterra type. Its proof will be given in the next section.
\begin{theorem}\label{thm:integreq}
Assume that $\mu$ is atom-less. Then for any $T>0$ the couple $(s_+,s_-)$ is the unique couple of left-continuous, increasing, positive functions that solve the system
\begin{align}\label{eq:intEq}
\int_t^T\int_{-s_-(T-u)}^{s_+(T-u)}p(t,\pm s_\pm(T-t),u,y)\big(\nu-\mu\big)(dy)\,du=0,\quad t\in[0,T)
\end{align}
with initial conditions $s_\pm(0)=\hat{b}_\pm$. Equivalently we may express \eqref{eq:intEq} in terms of $\varphi$ as
\begin{align}\label{eq:intEq2}
\int_\RR\mathds{1}_{\{y\,:\,\varphi(x)>\varphi(y)\}}\EE_x L^y_{\varphi(x)-\varphi(y)}\big(\nu-\mu\big)(dy)=0,\quad x\in \RR.
\end{align}
\end{theorem}
\begin{remark}
It is important to notice that \eqref{eq:intEq} reduces to a single equation if $\supp\{\mu\}\cap\RR_-=\emptyset$ (hence $s_-(\cdot)=+\infty$) or $\supp\{\mu\}\cap\RR_+=\emptyset$ (hence $s_+(\cdot)=+\infty$). Figures 2 and \ref{fig:3} illustrate numerical solutions of \eqref{eq:intEq} also in cases with a single boundary. 
\end{remark}
\vspace{+5pt}

Since the proof of the above theorem relies on the analysis of a problem of optimal stopping, then we introduce such problem here before closing this section. For $0<T<+\infty$ and $(t,x)\in[0,T]\times\RR$ we denote
\begin{align}
\label{eq:G}G(x):=&2\int_0^x\big(F_\nu(z)-F_\mu(z)\big)dz
\end{align}
and define
\begin{align}
\label{eq:V}V^T(t,x):=&\sup_{0\le \tau\le T-t}\EE_x G(B_\tau)
\end{align}
where the supremum is taken over all $(\cF_t)$-stopping times in $[0,T-t]$. Here the upper index $T$ in the definition of the value function is used to account for the time horizon explicitly. As usual the continuation set $\CC_T$ and the stopping set $\DD_T$ of \eqref{eq:V} are given by
\begin{align}
\label{eq:CC}&\CC_T:=\{(t,x)\in[0,T]\times\RR\,:\,V^T(t,x)>G(x)\}\\[+3pt]
\label{eq:DD}&\DD_T:=\{(t,x)\in[0,T]\times\RR\,:\,V^T(t,x)=G(x)\}
\end{align}
and clearly 
\begin{align*}
\DD_T=([0,T]\times\RR)\setminus\CC_T.
\end{align*}
Throughout the paper we will often use the following notation: for a set $A\subset [0,T]\times \RR$ we denote $A\cap\{t<T\}:=\{(t,x)\in A:t<T\}$.

Problem \eqref{eq:V} was addressed in \cite{DeA15-SK} where it was shown that the geometry of the continuation and stopping sets can be characterised in terms of the boundaries $s_\pm$ of the reversed Rost's barrier. Here we summarise those results in the following theorem under assumptions (D.1) and (D.2).
\begin{theorem}[De Angelis \cite{DeA15-SK}]\label{them:OS}
For each $T>0$ it holds
\begin{align}
\label{eq:CC-2}&\CC_T=\big\{(t,x)\in[0,T)\times\RR \,:\,x\in \big(-s_-(T-t),s_+(T-t)\big)\big\}
\end{align}
and the smallest optimal stopping time of problem \eqref{eq:V} is given by
\begin{align}\label{eq:optst}
\tau^T_*(t,x)=\inf\big\{s\in[0,T-t]\,:\,(t+s,x+B_s)\notin\CC_T\big\}
\end{align}
for $(t,x)\in[0,T]\times\RR$. Moreover the functions $s_\pm$ fulfill the following additional properties: $s_\pm(T-t)\ge \hat b_\pm$ for $t\in[0,T)$ and $s_\pm(0+)=\hat{b}_\pm$.
\end{theorem}

For $T>0$ it is convenient to introduce the function
\begin{align}\label{eq:UT}
U^T(t,x):=V^T(t,x)-G(x),\quad (t,x)\in[0,+\infty)\times\RR.
\end{align}
Applying It\^o-Tanaka-Meyer formula to $G$ inside the expectation of \eqref{eq:V} we find an alternative representation for $U^T$ which will also be useful in the rest of the paper, i.e.
\begin{align}\label{eq:U2}
U^T(t,x)=\sup_{0\le \tau\le T-t}\int_\RR\EE_x L^y_\tau\,(\nu-\mu)(dy).
\end{align}

In order to prove that the couple $(s_+,s_-)$ solves the integral equations \eqref{eq:intEq} we use a characterisation of the function $U^T$ which is of independent interest. We provide this in a separate proposition whose proof will be given in the next section.
\begin{prop}\label{prop:integreq}
Assume that $\mu$ is atom-less. Then for any $T>0$ the following equivalent representations of \eqref{eq:UT} hold
\begin{align}
\label{eq:intU}U^T(t,x)=&\int_t^T\int_{-s_-(T-u)}^{s_+(T-u)}p(t,x,u,y)\big(\nu-\mu\big)(dy)\,ds\\[+3pt]
\label{eq:intU2}U^T(t,x)=&\int_\RR\mathds{1}_{\{y\,:\,\varphi(y)<T-t\}}\EE_x L^y_{T-t-\varphi(y)}\big(\nu-\mu\big)(dy)
\end{align}
for $(t,x)\in[0,T]\times\RR$.
\end{prop}


\section{Proof of the main theorem}\label{sec:proof}

In this section we provide the proof of Theorem \ref{thm:integreq}. This requires to show first that the value function $V^T$ of \eqref{eq:V} is $C^1$ in the whole space and then to introduce a free boundary problem. The latter allows us to obtain the representation formulae of Proposition \ref{prop:integreq} which finally lead to the result in Theorem \ref{thm:integreq}. These steps will be taken in three separate sub-sections.

\subsubsection*{3.1~~$C^1$ regularity of the value function and a free boundary problem}
We start by recalling that Proposition 3.1 in \cite{DeA15-SK} guarantees that
\begin{align}\label{eq:Vcont}
V^T\in C([0,T]\times\RR).
\end{align} 
Moreover it has been shown in \cite{DeA15-SK}, and it is one of the paper's main technical results, that the time-derivative $V^T_t$ of the value function is indeed continuous on $[0,T)\times\RR$. However it is also mentioned there that when $\mu$ has atoms the continuity of $V^T_x(t,\cdot)$ could break down across the boundaries $s_\pm$ (let alone continuity of $V^T_x(\cdot,\cdot)$). For this reason here we focus only on continuous target measures $\mu$.

First we summarise the results of Lemma 3.7 of \cite{DeA15-SK} which will be needed in the following Proposition.
\begin{lemma}\label{lem:reg-bd}
For any $(t,x)\in \partial\CC_T\cap\{t<T\}$ and any sequence $(t_n,x_n)_{n}\subset \CC_T$ such that $(t_n,x_n)\to (t,x)$ as $n\to\infty$ it holds
\begin{align}\label{lim2}
\lim_{n\to0} \tau^T_*(t_n,x_n)=0,\qquad \PP-a.s.
\end{align}
\end{lemma}
Next we provide the $C^1$ regularity of $V^T$. 
\begin{prop}\label{prop:sm-fit}
Assume that $\mu$ is atom-less. Then for any $T>0$ it holds $V^T\in C^1([0,T)\times\RR)$.
\end{prop}
\begin{proof}
$(a).$ As already mentioned one can check \cite{DeA15-SK} for $V^T_t\in C([0,T)\times\RR)$.
\vspace{+5pt}

$(b).$ Here we prove that $V^T_x\in C([0,T)\times\RR)$. The claim is obvious in the interior of the continuation set $\CC_T$ where it is well known that $V^T\in C^{1,2}$ (see \cite[Sec.~7.1]{Pes-Shir}). In the interior of the stopping set $\DD_T$ one has $V^T=G$ and $V^T_x=G'=2(F_\nu-F_{\mu})$. The latter is continuous in $\DD_T\cap\{t<T\}$ because for all $t\in[0,T)$ one has $s_\pm(t)>\hat{b}_\pm$ due to Corollary \ref{cor:flat}, and therefore
\begin{align*}
\text{$F_\nu(x)=0$~for $x\le -s_-(t)$ and $F_\nu(x)=1$ for $x\ge s_+(t)$,}
\end{align*} 
for all $t\in[0,T)$ such that $s_+(t)<+\infty$ and/or $s_-(t)<+\infty$.

Then we must only consider points along the two boundaries. Let $(t,x)\in\partial\CC_T\cap\{t<T\}$ be arbitrary but fixed and let $(t_n,x_n)_n\subset \CC_T$ be a sequence of points converging to $(t,x)$ as $n\to \infty$. Fix an arbitrary $n\in\NN$ and for simplicity denote $\tau_n:=\tau^T_*(t_n,x_n)$ (see \eqref{eq:optst}). 

Since stopping at $\tau_n$ is optimal for the initial conditions $(t_n,x_n)$, then for $h>0$ such that $(t_n,x_n+h)\in\CC_T$ we use the mean value theorem to obtain 
\begin{align}
\label{Vx1}&\frac{V^T(t_n,x_n+h)-V^T(t_n,x_n)}{h}\ge \frac{1}{h}\EE\Big[G(X^{x_n+h}_{\tau_{n}})-G(X^{x_n}_{\tau_{n}})\Big]=\EE G'(\xi^h_1)
\end{align}
with $\xi^h_1\in(X^{x_n}_{\tau_n},X^{x_n+h}_{\tau_n})$, $\PP$-a.s. On the other hand by the same arguments we also obtain
\begin{align}
\label{Vx2}&\frac{V^T(t_n,x_n)-V^T(t_n,x_n-h)}{h}\le \frac{1}{h}\EE\Big[G(X^{x_n}_{\tau_{n}})-G(X^{x_n-h}_{\tau_{n}})\Big]=\EE G'(\xi^h_2)
\end{align}
with $\xi^h_2\in(X^{x_n-h}_{\tau_{n}},X^{x_n}_{\tau_{n}})$, $\PP$-a.s. 

From continuity of the sample paths we get that $X^{x_n+h}_{\tau_n}$ and $X^{x_n-h}_{\tau_{n}}$ converge to $X^{x_n}_{\tau_n}$ as $h\to 0$, $\PP$-a.s., hence $\xi^h_1,\,\xi^h_2\to X^{x_n}_{\tau_n}$ as $h\to 0$ as well. Recall that $V_x$ exists inside $\CC_T$ and $G$ is Lipschitz on $\RR$, so that dominated convergence and \eqref{Vx1} and \eqref{Vx2} allow us to conclude
\begin{align}\label{Vx3}
V^T_x(t_n,x_n)=\EE G'(X^{x_n}_{\tau_n}).
\end{align}
Now we use continuity of sample paths and Lemma \ref{lem:reg-bd} to obtain $X^{x_n}_{\tau_n}\to x$ as $n\to\infty$. Hence another application of dominated convergence and \eqref{Vx3} give
\begin{align}
\lim_{n\to\infty}V^T_x(t_n,x_n)=G'(x).
\end{align}
Since the point $(t,x)$ and the sequence $(t_n,x_n)$ were arbitrary we conclude that $V^T_x$ is continuous across the boundaries and hence everywhere in $[0,T)\times \RR$.
\end{proof}

It follows from the standard theory (see for instance \cite[Sec.~7.1]{Pes-Shir}) that \eqref{eq:Vcont} and the strong Markov property imply that $V^T\in C^{1,2}$ in $\CC_T$ and it solves
\begin{align}
\label{eq:pdeV01}& \big(V^T_t+\tfrac{1}{2}V^T_{xx}\big)(t,x)=0, &(t,x)\in \CC_T\\[+3pt]
\label{eq:pdeV02}& V^T(t,x)=G(x), &(t,x)\in\partial\CC_T
\end{align}
where $\CC_T$ may be expressed in terms of the curves $s_\pm (T-\,\cdot)$ as in \eqref{eq:CC-2}. It now follows from Proposition \ref{prop:sm-fit} that if $\mu$ is atom-less then $V^T$ is also globally $C^1$ in $[0,T)\times\RR$ and two more boundary conditions hold as well, i.e.
\begin{align}
\label{eq:pdeV03}& V^T_x(t,x)=G'(x), &(t,x)\in\partial\CC_T\cap\{t<T\}\\[+3pt]
\label{eq:pdeV03b}& V^T_t(t,x)=0, &(t,x)\in\partial\CC_T\cap\{t<T\}.
\end{align}

It is important to notice that $V^T_t$ is continuous on $[0,T)\times\RR$ even without assuming that $\mu$ is atom-less (see \cite[Prop.~3.15]{DeA15-SK}). Then $V^T_{xx}$ is continuous on the closure of the set $\CC_T\cap\{t<T-\eps\}$ for any $\eps>0$, due to \eqref{eq:pdeV01} and continuity of $V^T_t$. In $\DD_T$ we have $V^T=G^T$ and therefore $V^T_{xx}(t,dx)=G''(dx)$ exists as a measure therein. 

The above considerations imply that the function $U^T=V^T-G$ (see \eqref{eq:UT}) is continuously differentiable in $t$ and $U^T_{xx}(t,\,\cdot\,)$ defines a measure on $\RR$ for each fixed $t\in[0,T)$. Notice that if $\mu$ is atom-less then thanks to \eqref{eq:pdeV03} such measure charges no mass at $\partial\CC_T\cap\{t<T\}$, hence the above boundary value problem can be stated for $U^T$ as 
\begin{align}
\label{eq:pdeU01}& \big(U^T_t+\tfrac{1}{2}U^T_{xx}\big)(t,x)=-(\nu-\mu)(dx), &(t,x)\in \CC_T\\[+3pt]
\label{eq:pdeU02}& U^T(t,x)=0, &(t,x)\in\partial\CC_T\\[+3pt]
\label{eq:pdeU03}& U^T_x(t,x)=U^T_t(t,x)=0, &(t,x)\in\partial\CC_T\cap\{t<T\}
\end{align}
where \eqref{eq:pdeU01} holds in the sense of distributions.

\subsubsection*{3.2~~A representation of the value function}

Here we prove Proposition \ref{prop:integreq}. The integral representation \eqref{eq:intU} for $U^T$ easily yields the integral equations for the boundaries $s_\pm$ and it is therefore instrumental for our main result, i.e.~Theorem \ref{thm:integreq}. Moreover some of the concepts introduced in the proof below will be needed when proving the latter theorem. 

\begin{proof}[Proof of Proposition \ref{prop:integreq}]
For the sake of this proof it is convenient to denote $b_\pm(t):=s_\pm(T-t)$ for $t\in[0,T]$ and use $b_\pm$ instead of $s_\pm$. Similarly we drop the apex $T$ and denote $U=U^T$.

The representation \eqref{eq:intU} for $U^T$ is obtained by an application of an extension of It\^o-Tanaka-Meyer formula to the time-space framework which is inspired by \cite{AJKY98}. However since the result of \cite[Prop.~4]{AJKY98} cannot be applied directly to our setting we provide a full proof for completeness.

We start by recalling that for each $t\in[0,T)$ the map $x\mapsto U_{x}(t,x)$ may be seen as the cumulative distribution function of a signed measure on $\RR$ with support on $[-b_-(t),b_+(t)]$. In fact we set $U_{xx}(t,(a,b]):=U_x(t,b)-U_x(t,a)$ for any interval $(a,b)\subset\RR$ (recall that $F_\nu$ is only right continuous) and we notice that due to \eqref{eq:pdeU03} $U_{xx}(t,\{\pm b_\pm(t)\})=0$ for all $t\in[0,T)$, i.e.~$U_{xx}(t,dx)$ has no atoms at the boundary points $b_\pm(t)$. It then follows
\begin{align}\label{eq:ito00}
U_{xx}(t,dx)=\mathds{1}_{\{x\in(-b_-(t),b_+(t))\}}\Big(V_{xx}(t,x)dx-2(\nu-\mu)(dx)\Big),\quad (t,x)\in[0,T)\times\RR.
\end{align}

We continuously extend $U$ to $\RR^2$ by taking $U(t,x)=U(0,x)$ for $t<0$ and $U(t,x)=U(T,x)=0$ for $t>T$. Then we pick positive functions $f,g\in C^\infty_c(\RR)$ such that
\begin{align*}
\int_\RR f(x)dx=\int_\RR g(x)dx=1
\end{align*}
and for $n\in\NN$ we set $f_n(x):=nf(nx)$, $g_n(x)=ng(nx)$. We use standard mollification to obtain a double indexed sequence of functions $(U_{n,m})_{n,m\in\NN^2}\subset C^\infty_c(\RR^2)$ defined by
\begin{align}\label{eq:ito01}
U_{n,m}(t,x):=\int_\RR\int_\RR U(s,y)f_n(x-y)g_m(t-s)\,dy\,ds,\quad (t,x)\in\RR^2.
\end{align}
To avoid technicalities related to the endpoints of the time domain fix $(t,x)\in (\eps,T-\eps)\times\RR$, $\eps>0$ arbitrary, then by applying classical Dynkin's formula for each $n$ and $m$ we find the expression
\begin{align}\label{eq:ito02}
U_{n,m}(t,x)=\EE\Big[U_{n,m}(T-\eps,B^{t,x}_{T-\eps})-\int_t^{T-\eps}\big(\tfrac{\partial}{\partial\,t} U_{n,m}+\tfrac{1}{2}\tfrac{\partial^2}{\partial\,x^2} U_{n,m}\big)(s,B^{t,x}_s)ds\Big].
\end{align}
By the standard properties of mollifiers it is not hard to see that taking limits as $n,m\to\infty$ and as $\eps\to0$ one has $U_{n,m}(t,x)\to U(t,x)$, $\EE\big[U_{n,m}(T-\eps,B^{t,x}_{T-\eps})]\to \EE\big[U(T,B^{t,x}_{T})]=0$.

It only remains to analyse the integral term.
Note that when extending $U$ to $\RR^2$ we have implicitly taken $b_\pm(t)=b_\pm(0)$ for $t<0$ and $b_\pm(t)=0$ for $t>T$. Note as well that there is no loss of generality assuming that the index $m$ in \eqref{eq:ito02} is large enough to have $g_m$ supported in $(-\eps,+\eps)$. Then for $(s,z)\in(t,T-\eps)\times\RR$ it holds
\begin{align}
\big(\tfrac{\partial^2}{\partial\,x^2}U_{n,m}\big)(s,z)=&\int_\RR\int^{b_+(v)}_{-b_-(v)}U(v,y)f''_n(z-y)g_m(s-v)\,dy\,dv\\
\big(\tfrac{\partial}{\partial\,t}U_{n,m}\big)(s,z)=&\int_\RR\int^{b_+(v)}_{-b_-(v)}U(v,y)f_n(z-y)g'_m(s-v)\,dy\,dv
\end{align}
where the integrals in $dv$ must only be taken for $v\in(s-\eps,s+\eps)\subset (0,T)$.

Use integration by parts twice, the smooth-fit \eqref{eq:pdeU03}, and \eqref{eq:ito00} to obtain
\begin{align}\label{eq:ito03}
\big(\tfrac{\partial^2}{\partial\,x^2}U_{n,m}\big)(s,z)
=&\int_\RR\Big(\int^{b_+(v)}_{-b_-(v)}f_n(z-y)U_{xx}(v,dy)\Big)g_m(s-v)\,dv\\
=&\int_\RR\Big(\int^{b_+(v)}_{-b_-(v)}f_n(z-y)\big(V_{xx}(v,y)dy-2(\nu-\mu)(dy)\big)\Big)g_m(s-v)\,dv\nonumber
\end{align}
and similarly
\begin{align}\label{eq:ito04}
\big(\tfrac{\partial}{\partial\,t}U_{n,m}\big)(s,z)=\int_\RR\int^{b_+(v)}_{-b_-(v)}U_t(v,y)f_n(z-y)g_m(s-v)\,dy\,dv.
\end{align}
Since $U_t=V_t$ and \eqref{eq:pdeV01} holds, then adding up \eqref{eq:ito03} and \eqref{eq:ito04} (where now $z=B^{t,x}_s$) the above gives
\begin{align}\label{eq:ito05}
\EE&\Big[\int_t^{T-\eps}\big(\tfrac{\partial}{\partial\,t} U_{n,m}+\tfrac{1}{2}\tfrac{\partial^2}{\partial\,x^2} U_{n,m}\big)(s,B^{t,x}_s)ds\Big]\\
=&-\EE\Big[\int_t^{T-\eps}\Big\{\int_\RR\Big(\int^{b_+(v)}_{-b_-(v)}f_n(B^{t,x}_s-y)(\nu-\mu)(dy)\Big)g_m(s-v)\,dv\Big\}ds\Big]\nonumber\\
=&-\int_\RR\int_t^{T-\eps}\hspace{-5pt}\Big\{\int_\RR\Big(\int^{b_+(v)}_{-b_-(v)}f_n(z-y)(\nu-\mu)(dy)\Big)g_m(s-v)\,dv\Big\}p(t,x,s,z)ds\,dz.\nonumber
\end{align}
Recall that $\supp\,g_m\subset (-\eps,+\eps)$. For $(s,z)\in(t,T-\eps)\times\RR$ we set
\begin{align}
\label{eq:lam00}&\Lambda_n(v,z):=\int^{b_+(v)}_{-b_-(v)}f_n(z-y)(\nu-\mu)(dy),\quad v\in (s-\eps,s+\eps)\\
\label{eq:lam01}&\Lambda_{n,m}(s,z):=\int_\RR\Lambda_n(v,z)g_m(s-v)dv
\end{align}
so that we rewrite \eqref{eq:ito05} in the form
\begin{align}\label{eq:ito06}
\EE\Big[\int_t^{T-\eps}\big(\tfrac{\partial}{\partial\,t} U_{n,m}+\tfrac{1}{2}\tfrac{\partial^2}{\partial\,x^2} U_{n,m}\big)(s,B^{t,x}_s)ds\Big]=-\int_\RR\int_t^{T-\eps}\Lambda_{n,m}(s,z)p(t,x,s,z)ds\,dz.
\end{align}

For fixed $n$ and for any $z\in\RR$ and $s\in(t,T-\eps)$ the map $v\mapsto \Lambda_n(v,z)$ is bounded and continuous on $(s-\eps,s+\eps)$. This follows since $\mu$ is flat at possible jumps of $b_\pm$ (cf.~\eqref{eq;jumpb}, \eqref{eq;jumpb2} of Theorem 2.1 and recall that $b_\pm(\cdot)=s_\pm(T-\cdot)$) and by observing that 
\begin{align*}
\int^{b_+(v)}_{-b_-(v)}f_n(z-y)\nu(dy)=\int^{a_+}_{-a_-}f_n(z-y)\nu(dy).
\end{align*}
Notice also that $g_m(s-\,\cdot\,)$ converges weakly to $\delta(s-\cdot)$ as $m\to \infty$ and hence for all $z\in\RR$ and $s\in(t,T-\eps)$ we get
\begin{align}\label{eq:lam02}
\lim_{m\to\infty}\Lambda_{n,m}(s,z)=\Lambda_n(s,z).
\end{align}

For fixed $n$ it holds $\|f_n\|_\infty\le K_n$ for suitably large $K_n>0$ and hence it also holds $\big|\Lambda_{n,m}(s,z)\big|\le 2K_n$ uniformly for $(s,z)\in(t,T-\eps)\times\RR$. We keep $n$ fixed and take limits as $m\to\infty$, then dominated convergence and \eqref{eq:lam02} give
\begin{align}\label{eq:lam03}
\lim_{m\to\infty}\int_\RR\int_t^{T-\eps}\Lambda_{n,m}(s,z)p(t,x,s,z)ds\,dz=\int_\RR\int_t^{T-\eps}\Lambda_{n}(s,z)p(t,x,s,z)ds\,dz.
\end{align}

In order to take limits as $n\to\infty$ we analyse the the right-hand side of the expression above. In particular we recall from \eqref{eq:invbarrier00} that
\begin{align}\label{eq:equiv}
y\in (-b_-(s),b_+(s))\Leftrightarrow y\in (-s_-(T-s),s_+(T-s))\Leftrightarrow \varphi(y)<T-s.
\end{align}
Hence Fubini's theorem and the occupation formula for local time give
\begin{align}\label{eq:lam04}
\int_\RR\int_t^{T-\eps}&\Lambda_{n}(s,z)p(t,x,s,z)ds\,dz\\
=&\EE_{t,x}\Big[\int_t^{T-\eps}\hspace{-5pt}\int_{-b_-(s)}^{b_+(s)}f_n(B_s-y)(\nu-\mu)(dy)\,ds\Big]\nonumber\\
=&\int_\RR\EE_{t,x}\Big[\int_t^{T-\eps}\mathds{1}_{\{s<T-\varphi(y)\}}f_n(B_s-y)ds\Big](\nu-\mu)(dy)\nonumber\\
=&\int_\RR\EE_{t,x}\Big[\int_t^{(T-\eps)\wedge([T-\varphi(y)]\vee t)}f_n(B_s-y)ds\Big](\nu-\mu)(dy)\nonumber\\
=&\int_\RR\Big(\int_\RR f_n(z-y)\EE_{t,x}\big[L^z_{(T-\eps)\wedge([T-\varphi(y)]\vee t)}\big]dz\Big)(\nu-\mu)(dy)\nonumber
\end{align}
where we recall that $(L^z_s)_{s\ge0}$ is the local time of the Brownian motion at a point $z\in\RR$.

With no loss of generality and thanks to the definition of $f_n$, we may assume that there exists a compact $K\subset\RR$ such that $\supp\, f_n\subset K$ for all $n$. For fixed $t$ and $x$ the quantity $\EE_{t,x}\big[L^z_{(T-\eps)\wedge([T-\varphi(y)]\vee t)}\big]$ is bounded uniformly with respect to all $y,z\in\RR$ such that $y-z\in K$ (use for example It\^o-Tanaka's formula to verify the claim). Moreover, for fixed $t$, $x$ and $y$ the map $z\mapsto \EE_{t,x}\big[L^z_{(T-\eps)\wedge([T-\varphi(y)]\vee t)}\big]$ is bounded and continuous for $z\in\RR$ such that $z-y\in K$. Finally, for each $y\in\RR$ one has $f_n(\cdot-y)\to \delta(\cdot - y)$ weakly as $n\to \infty$. 

The above considerations allow us to use dominated convergence (with respect to the measure $\nu-\mu$) and \eqref{eq:lam04} to obtain 
\begin{align}\label{eq:lam05}
\lim_{n\to\infty} \int_\RR\int_t^{T-\eps}&\Lambda_{n}(s,z)p(t,x,s,z)ds\,dz=\int_\RR\EE_{t,x}\big[L^y_{(T-\eps)\wedge([T-\varphi(y)]\vee t)}\big](\nu-\mu)(dy).
\end{align}

Now, we collect \eqref{eq:ito02} and \eqref{eq:ito06}, we take limits as $m\to\infty$ first, then as $n\to\infty$ and finally as $\eps\to 0$. Then using \eqref{eq:lam03} and \eqref{eq:lam05} we obtain
\begin{align}\label{eq:int00}
U(t,x)=&\int_\RR\EE_{t,x}\big[L^y_{[T-\varphi(y)]\vee t}\big](\nu-\mu)(dy)=\int_\RR \mathds{1}_{\{y:\varphi(y)<T-t\}}\EE_{t,x}\big[L^y_{T-\varphi(y)}\big](\nu-\mu)(dy)\\
=&\int_\RR \mathds{1}_{\{y:\varphi(y)<T-t\}}\EE_{x}\big[L^y_{T-t-\varphi(y)}\big](\nu-\mu)(dy).\nonumber
\end{align}
The equivalent expression \eqref{eq:intU} can be obtained by recalling the well known representation for the expectation of the Brownian local time
\begin{align}\label{eq:Eloct}
\EE_{x}\big[L^y_{T-t-\varphi(y)}\big]
=\int_t^{T-\varphi(y)}p(t,x,s,y)ds.
\end{align}
Then by using once again \eqref{eq:equiv} we obtain
\begin{align}\label{eq:int01}
U(t,x)=&\int_\RR\Big(\int_t^T\mathds{1}_{\{s<T-\varphi(y)\}}p(t,x,s,y)ds\Big)(\nu-\mu)(dy)\\
=&\int_\RR\Big(\int_t^T\mathds{1}_{\{y\in(-b_-(s),b_+(s))\}}
p(t,x,s,y)ds\Big)(\nu-\mu)(dy)\nonumber
\end{align}
and hence \eqref{eq:intU} follows by a simple application of Fubini's theorem and by putting $b_\pm(t)=s_\pm(T-t)$.
\end{proof}

\subsubsection*{3.3~~Uniqueness of the solution for the integral equations}

We can now prove Theorem \ref{thm:integreq}. 

\begin{proof}[Proof of Theorem \ref{thm:integreq}]
We develop the proof considering a setting with two boundaries, i.e.~$\mu$ supported on both $\RR_-$ and $\RR_+$. A simpler version of the same arguments holds in the cases with a single boundary.

First we notice that the system \eqref{eq:intEq} of integral equations is easily obtained by taking $x=\pm s_\pm(T-t)$, $t<T$ in \eqref{eq:intU} and observing that $U^T(t,\pm s_\pm(T-t))=0$. 
For the rest of this proof it is convenient to denote $b_\pm(t):=s_\pm(T-t)$ for $t\in[0,T]$ and use $b_\pm$ instead of $s_\pm$. Also we drop the apex $T$ and denote $U=U^T$.

The proof follows a contradiction argument. Let us assume that there exists a couple of functions $(r_-,r_+)$ with the following properties: $r_\pm:\RR_+\to\RR_+\cup\{+\infty\}$, increasing, left-continuous and such that $(r_-,r_+)$ solves \eqref{eq:intEq} with terminal conditions $r_\pm(0+)=\hat{b}_\pm$. Notice that by monotonicity $r_\pm(T-t)\ge a_\pm$ for $t\in[0,T)$. As in \eqref{eq:invbarrier00} we can define a generalised inverse of $r_\pm$ and we denote it by $\psi:\RR\to[0,\infty)$.

Again it is convenient to consider a time reversal and denote $c_\pm(t):=r_\pm(T-t)$, $t\in[0,T]$. Motivated by \eqref{eq:intU2} let us define a function
\begin{align}\label{eq:Uc00}
U^c(t,x):=\int_\RR\EE_xL^y_{(T-t-\psi(y))^+}(\nu-\mu)(dy),\quad(t,x)\in[0,T]\times\RR
\end{align}
and notice that for $\psi=\varphi$ we simply have $U^c=U$ (observe that $\mathds{1}_{\{y:\psi(y)<T-t\}}L^y_{T-t-\psi(y)}=L^y_{(T-t-\psi(y))^+}$). By assumption $(c_-,c_+)$ solves \eqref{eq:intEq} and hence
\begin{align}\label{eq:intEqc}
U^c(t,\pm c_\pm(t))=0,\quad t\in[0,T),\:\:\:\text{and}\:\:\: U^c(T,x)=0, \quad x\in\RR.
\end{align}

Using that $\EE_xL^y_{(T-t-\psi(y))^+}=\EE_x\big|B_{(T-t-\psi(y))^+}-y\big|-|x-y|$ it is not difficult to prove that $U^c\in C([0,T]\times\RR)$. Now we want to show that the process $Y^c$ defined as
\begin{align}\label{eq:Yc}
Y^c_s:=U^c(t+s,B_s)+\int_\RR L^y_{s\wedge(T-t-\psi(y))^+}(\nu-\mu)(dy),\quad s\in[0,T-t]
\end{align}
is a $\PP_x$-martingale for any $x\in\RR$ and $t\in[0,T)$. By the definition \eqref{eq:Uc00} it follows
\begin{align}\label{eq:Uc01}
\EE_x U^c(t+s,B_s)=\EE_x\Big[\int_\RR\mathds{1}_{\{t+s<T-\psi(y)\}}\EE_{B_s} \big[L^y_{(T-(t+s)-\psi(y))}\big](\nu-\mu)(dy)\Big].
\end{align}
Strong Markov property and It\^o-Tanaka's formula give
\begin{align}\label{eq:Uc02}
\EE_{B_s} L^y_u=\EE_x\Big[\big|B_{s+u}-y\big|\Big|\cF_s\Big]-\big|B_s-y\big|=\EE_x\big(L^y_{s+u}\big|\cF_s\big)-L^y_s-M_s, \quad\PP_x-\text{a.s.}
\end{align}
for all $u\ge 0$ and with $(M_s)_{s\ge0}$ a $\PP_x$-martingale with $M_0=0$, $\PP_x$-a.s. In particular for $u=T-(t+s)-\psi(y)$, by Fubini's theorem and \eqref{eq:Uc02} we obtain
\begin{align}\label{eq:Uc03}
\EE_x U^c(t+s,B_s)=\int_\RR\mathds{1}_{\{s<T-t-\psi(y)\}}\EE_x\big[L^y_{(T-t-\psi(y))}-L^y_s\big](\nu-\mu)(dy).
\end{align}
Now we add and subtract $U^c(t,x)$ from the right-hand side of the above expression, use \eqref{eq:Uc00} and simplify the indicator variables to get
\begin{align}\label{eq:Uc04}
\EE_x U^c(t+s,B_s)=&U^c(t,x)\nonumber\\[+3pt]
&+\int_\RR\big(\mathds{1}_{\{s<T-t-\psi(y)\}}-\mathds{1}_{\{0<T-t-\psi(y)\}}\big)\EE_x L^y_{(T-t-\psi(y))}(\nu-\mu)(dy)\nonumber\\[+3pt]
&-\int_\RR\mathds{1}_{\{s<T-t-\psi(y)\}}\EE_xL^y_s(\nu-\mu)(dy)\\[+3pt]
=&U^c(t,x)-\int_\RR\mathds{1}_{\{0<T-t-\psi(y)\le s\}}\EE_x L^y_{(T-t-\psi(y))}(\nu-\mu)(dy)\nonumber\\[+3pt]
&-\int_\RR\mathds{1}_{\{s<T-t-\psi(y)\}}\EE_xL^y_s(\nu-\mu)(dy)&\nonumber\\[+3pt]
=&U^c(t,x)-\EE_x\int_\RR L^y_{s\wedge(T-t-\psi(y))^+}(\nu-\mu)(dy).\nonumber
\end{align}
Therefore the process $Y^c$ is a continuous martingale as claimed and from now on we will refer to this property as to the \emph{martingale property of $U^c$}. Further, we notice that in the particular case of $c_\pm=b_\pm$ we obtain what we will refer to as the \emph{martingale property of $U$}. Then for any $(t,x)\in[0,T]\times\RR$ and any stopping time $\tau\in[0,T-t]$ it holds
\begin{align}
\label{mgpropU} U(t,x)=\EE_x\Big[U(t+\tau,B_\tau)+\int_\RR L^y_{\tau\wedge(T-t-\varphi(y))^+}(\nu-\mu)(dy)\Big],\\
\label{mgpropUc}U^c(t,x)=\EE_x\Big[U^c(t+\tau,B_\tau)+\int_\RR L^y_{\tau\wedge(T-t-\psi(y))^+}(\nu-\mu)(dy)\Big].
\end{align}

We now proceed in four steps inspired by Peskir \cite{Pe-0}. The equations considered by Peskir can be seen as a special case of \eqref{eq:intEq} and \eqref{eq:intEq2} and here we obtain uniqueness in a stronger sense, i.e.~in the class of right-continuous functions rather than continuous ones as in \cite{Pe-0}.
\vspace{+6pt}

\begin{figure}[!ht]
\label{fig:1}
\begin{center}
\includegraphics[scale=0.45]{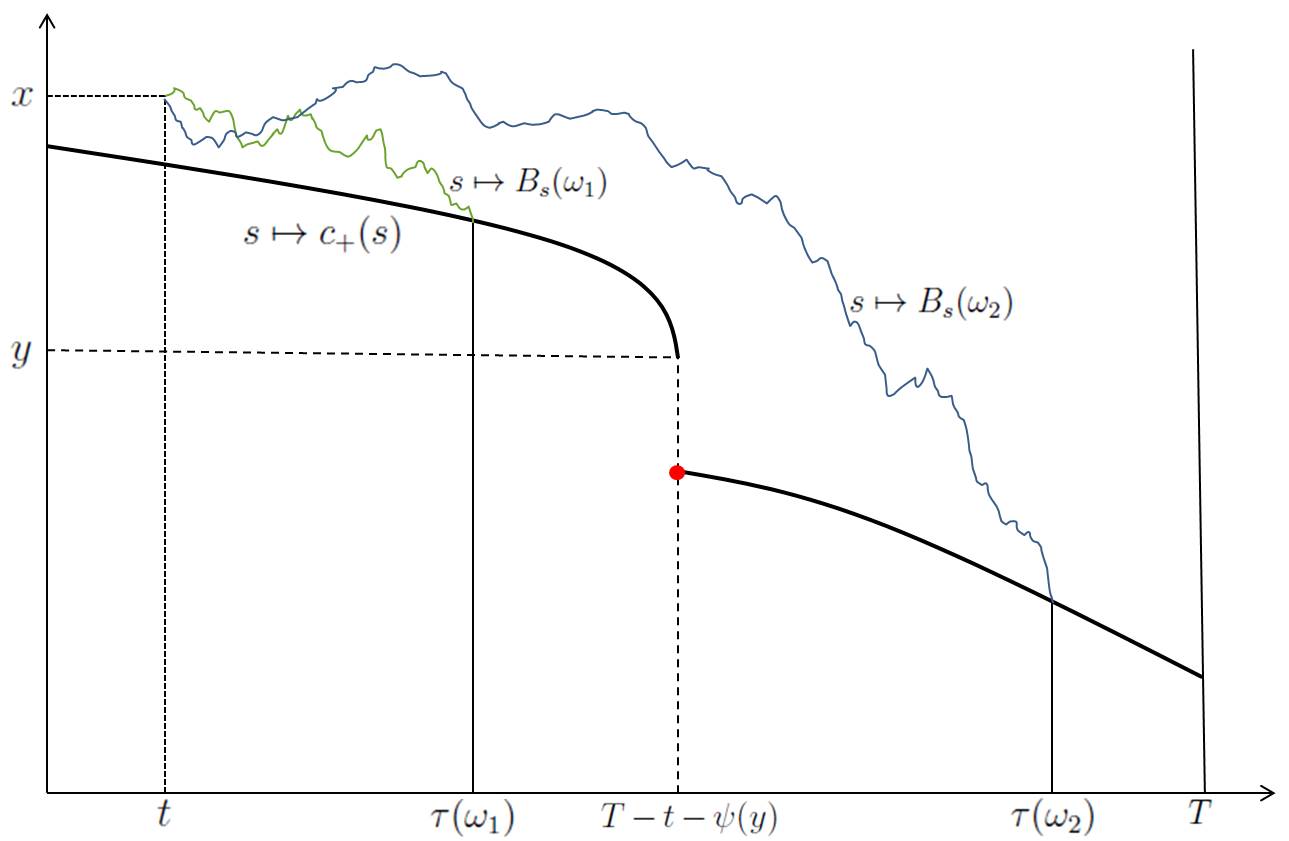}
\end{center}
\caption{Examples of Brownian sample paths prior to the first hitting time $\tau$ to the boundary $c_+$.}
\end{figure}

\emph{Step 1}. First we show that
\begin{align}\label{eq:Uc05}
U^c(t,x)=0,\quad\text{for $t\in[0,T)$, $x\in(-\infty,-c_-(t)]\cup[c_+(t),\infty)$}.
\end{align}
We fix $t\in[0,T)$ and with no loss of generality it is sufficient to consider $x>c_+(t)$ since the proof for $x<-c_-(t)$ follows from analogous arguments. We set $\tau:=\inf\{s\in(0,T-t]\,:\,B_s\le c_+(t+s)\}$ and note that since $s\mapsto B_s(\omega)-c_+(t+s)$ is positive for $s< \tau$, right-continuous, with only upwards jumps, then $B_s(\omega)-c_+(t+s)$ must cross zero in a continuous way and it must be $B_\tau=c_+(t+\tau)$, $\PP_x$-a.s. Then \eqref{eq:intEqc} implies $U^c(t+\tau,B_\tau)=0$, $\PP_x$-a.s.~and it follows from \eqref{eq:Uc04} that
\begin{align}\label{eq:Uc06}
U^c(t,x)=\int_\RR\EE_x L^y_{\tau\wedge(T-t-\psi(y))^+}(\nu-\mu)(dy).
\end{align}
The support of $y\mapsto L^y_{\tau\wedge(T-t-\psi(y))^+}$ is contained in $[\hat{b}_+,c_+(t))$, $\PP_x$-a.s.~by definition of $\tau$ and by observing that 
\begin{align}\label{interv}
\{z\in\RR:\psi(z)<T-t\}=(-c_-(t),c_+(t))
\end{align}
with $c_+(T)=\hat{b}_+$. Since $c_+$ is monotonic decreasing, for any $y>0$ it is not hard to see that:
\begin{align*}
&\text{~$i)$ if $\tau(\omega)\le T-t-\psi(y)$, then $B_s(\omega)> y$ for $s< \tau(\omega)$;}\\[+3pt]
&\text{$ii)$ if $\tau(\omega)> T-t-\psi(y)$, then $B_s(\omega)>y$ for $s< T-t-\psi(y)$;}
\end{align*}
hence $L^y_{\tau\wedge(T-t-\psi(y))^+}=0$, $\PP_x$-a.s.~and \eqref{eq:Uc05} follows (see also Figure 1).
Since $U$ is non negative by definition (see \eqref{eq:UT}) then it also follows
\begin{align}\label{eq:Uc07}
0=U^c(t,x)\le U(t,x),\quad\text{for $t\in[0,T)$, $x\in(-\infty,-c_-(t)]\cup[c_+(t),\infty)$}.
\end{align}
\vspace{+6pt}

\emph{Step 2}. Next we prove that $U^c\le U$ on $[0,T]\times\RR$. Since $U^c(T,x)=U(T,x)$, $x\in\RR$ and \eqref{eq:Uc07} holds it only remains to consider $t\in[0,T)$ and $x\in(-c_-(t),c_+(t))$. Fix any such $(t,x)$ and denote $\sigma:=\inf\{s\in[0,T-t)\,:\,B_s\le-c_-(t+s)\:\text{or}\:B_s\ge c_+(t+s)\}$. Notice that due to \eqref{eq:Uc05} and monotonicity of $c_+$ and $c_-$, we have $U^c(t+\sigma,B_\sigma)=0$, $\PP_x$-a.s.

Then from \eqref{eq:U2} and the martingale property of $U^c$ we obtain
\begin{align}\label{eq:Uc08}
U(t,x)-U^c(t,x)\ge& \int_\RR\EE_x\big(L^y_\sigma-L^y_{\sigma\wedge(T-t-\psi(y))^+}\big)(\nu-\mu)(dy)\\
\ge&-\int_\RR\EE_x\big(L^y_\sigma-L^y_{\sigma\wedge(T-t-\psi(y))^+}\big)\mu(dy)\nonumber
\end{align}
since $u\mapsto L^y_u(\omega)$ is increasing for each $y\in\RR$. We recall \eqref{interv} and claim that
\begin{align}\label{claim}
\text{$L^y_\sigma=L^y_{\sigma\wedge(T-t-\psi(y))^+}$, $\PP_x$-a.s.~for $y\in(-c_-(t),c_+(t))$.}
\end{align}
The latter is trivial if $\sigma(\omega)\le T-t-\psi(y)$. Consider instead $\sigma(\omega)>T-t-\psi(y)$ and with no loss of generality assume $y\ge \hat{b}_+$. Monotonicity of $c_+$ implies that $B_s(\omega)<y$ for $s\in(T-t-\psi(y),\sigma(\omega))$ (see also Figure 1) so that the local time cannot increase after time $T-t-\psi(y)$ and therefore the claim holds. The case of $y\le-\hat{b}_-$ follows from symmetric arguments relative to $c_-(\cdot)$.

From \eqref{claim} and observing that $y\mapsto L^y_\sigma$ is supported in $(-c_-(t),c_+(t))$, $\PP_x$-a.s.,~we conclude that the right hand side of \eqref{eq:Uc08} equals zero and~$U^c\le U$.
\vspace{+6pt}

\emph{Step 3.} Now we aim at proving that $c_\pm(t)\le b_\pm(t)$ for $t\in[0,T)$. Let us assume that there exists $t\in[0,T)$ such that $c_+(t)>b_+(t)$ and take $x>c_+(t)$. We denote $\tau':=\inf\{s\in[0,T-t)\,:\,B_s<b_+(t+s)\}$ and from the same arguments as those used in step 1~above we have $B_{\tau'}=b_+(t+\tau')$ and $U(t+\tau',B_{\tau'})=0$, $\PP_x$-a.s. Since $x>c_+(t)>b_+(t)$ we have also $U^c(t,x)=U(t,x)=0$ and from step 2~above it follows $0=U(t+\tau',B_{\tau'})\ge U^c(t+\tau',B_{\tau'})$, $\PP_x$-a.s. Subtracting \eqref{mgpropU} from \eqref{mgpropUc} now gives 
\begin{align*}
\int_\RR\EE_x\big(L^y_{\tau'\wedge(T-t-\psi(y))^+}-L^y_{\tau'\wedge(T-t-\varphi(y))^+}\big)(\nu-\mu)(dy)\ge 0.
\end{align*}
Due to monotonicity of $b_+$, an argument similar to the one that allowed us to show that the right-hand side of \eqref{eq:Uc06} vanishes, here implies $L^y_{\tau'\wedge(T-t-\varphi(y))^+}=0$, $\PP_x$-a.s. Therefore from the last inequality it follows
\begin{align}\label{eq:Uc09}
0\le \int_\RR\EE_x L^y_{\tau'\wedge(T-t-\psi(y))^+}(\nu-\mu)(dy)=-\int_\RR\EE_x L^y_{\tau'\wedge(T-t-\psi(y))^+}\mu(dy)\le0,
\end{align}
where we have also used that $y\in\supp\,\nu \Rightarrow L^y_{\tau'\wedge(T-t-\psi(y))^+}=0$, $\PP_x$-a.s.~since $x+B_s\ge a_+$ for $s\le\tau'$ (recall that $b_+(\cdot)>a_+$ in $[0,T)$).

Now \eqref{eq:Uc09} implies that
\begin{align}\label{zero}
\text{for $\mu$-a.e.~$y\in\RR$, $L^y_{\tau'\wedge(T-t-\psi(y))^+}=0$, $\PP_x$-a.s.}
\end{align}
Since $y\mapsto L^y_{\tau'\wedge(T-t-\psi(y))^+}$ is supported in $[\hat{b}_+,c_+(t))$, $\PP_x$-a.s.~and $x>c_+(t)$, then \eqref{zero} also implies that
\begin{align}\label{contr}
\text{$\tau'\wedge(T-t-\psi(y))^+\le\zeta_y$, $\PP_x$-a.s., for $\mu$-a.e.~$y\in\RR$}
\end{align}
with $\zeta_y:=\inf\{s\ge0\,:\,B_s= y\}$. Next we show that \eqref{contr} leads to a contradiction.

Let $I\subset (b_+(t), c_+(t))$ be an open interval. For any $y\in I$, monotonicity of $b_+$ implies that $\tau'\ge \zeta_y$, $\PP_x$-a.s.~and therefore it follows from \eqref{contr} that for $\mu$-a.e.~$y\in I$ it must be $(T-t-\psi(y))^+\le \zeta_y$, $\PP_x$-a.s. However, for any fixed $y\in I$, by right-continuity of $c_+$ there must exist $\eps>0$ such that
$c_+(t+s)> y$ for $s\in[0,\eps]$ and equivalently $\psi(y)<T-(t+\eps)$. For such $y$ and $\eps$ this leads to $(T-t-\psi(y))^+\le \zeta_y\Rightarrow \zeta_y>\eps$, $\PP_x$-a.s., which is clearly impossible since $\PP_x(\zeta_y\le\eps)>0$. Therefore $c_+(t)\le b_+(t)$ for all $t\in[0,T)$.

To prove that $c_-(t)\le b_-(t)$, $t\in[0,T)$ we assume that there exists $t\in[0,T)$ such that $c_-(t)>b_-(t)$; then we pick $x<-c_-(t)$ and follow arguments as above to show a contradiction.
\vspace{+6pt}

\emph{Step 4.} At this point we prove $c_\pm=b_\pm$. Again it is enough to give the full argument only for the upper boundaries as the case of the lower ones is similar.

We assume that there exists $t\in[0,T)$ such that $c_+(t)<b_+(t)$, then by definition of $b_+$ there must exists $\delta>0$ and $x\in(c_+(t),b_+(t))$ such that $U(t,x)\ge\delta$. We denote $\sigma':=\inf\{s\in[0,T-t)\,:\,B_s\le -b_-(t+s)\:\text{or}\:B_s\ge b_+(t+s)\}$ and we recall that $\sigma'$ is optimal for $U(t,x)$ (i.e.~$\sigma'=\tau_*$ as in \eqref{eq:optst}), so that \eqref{eq:U2} gives
\begin{align}\label{eq:U3}
U(t,x)=\int_\RR\EE_x L^y_{\sigma'}(\nu-\mu)(dy).
\end{align}
Since $c_\pm\le b_\pm$ (from step 3~above) and $U^c(T,x)=0$, it follows that $U^c(t,x)=0$ and $U^c(t+\sigma',B_{\sigma'})=0$, $\PP_x$-a.s.~by \eqref{eq:Uc05}. Then subtracting \eqref{mgpropUc} from \eqref{eq:U3} gives
\begin{align}\label{eq:Uc10}
\delta\le&\int_\RR\EE_x\big(L^y_{\sigma'}-L^y_{\sigma'\wedge(T-t-\psi(y))^+}\big)(\nu-\mu)(dy)\\
=&\int_\RR\EE_x\mathds{1}_{\{\sigma'\ge(T-t-\psi(y))^+\}}\big(L^y_{\sigma'}-L^y_{(T-t-\psi(y))^+}\big)(\nu-\mu)(dy).\nonumber
\end{align}
We observe that for $y\in\supp\,\nu$ one has $\psi(y)=0$ (recall $c_\pm(t)\ge a_\pm$, $t\in[0,T)$) and then for any such $y$ and $\PP_x$-a.e.~$\omega\in\{\sigma'\ge T-t-\psi(y)\}$ we have $\sigma'(\omega)=T-t$. Using the latter inside \eqref{eq:Uc10} we get
\begin{align}\label{eq:Uc11}
\delta\le-\int_\RR\EE_x\mathds{1}_{\{\sigma'\ge(T-t-\psi(y))^+\}}\big(L^y_{\sigma'}-L^y_{(T-t-\psi(y))^+}\big)\mu(dy)\le 0
\end{align}
which leads to a contradiction.
\end{proof}

\begin{figure}[!ht]
\label{fig:2}
\begin{center}
\includegraphics[scale=0.65]{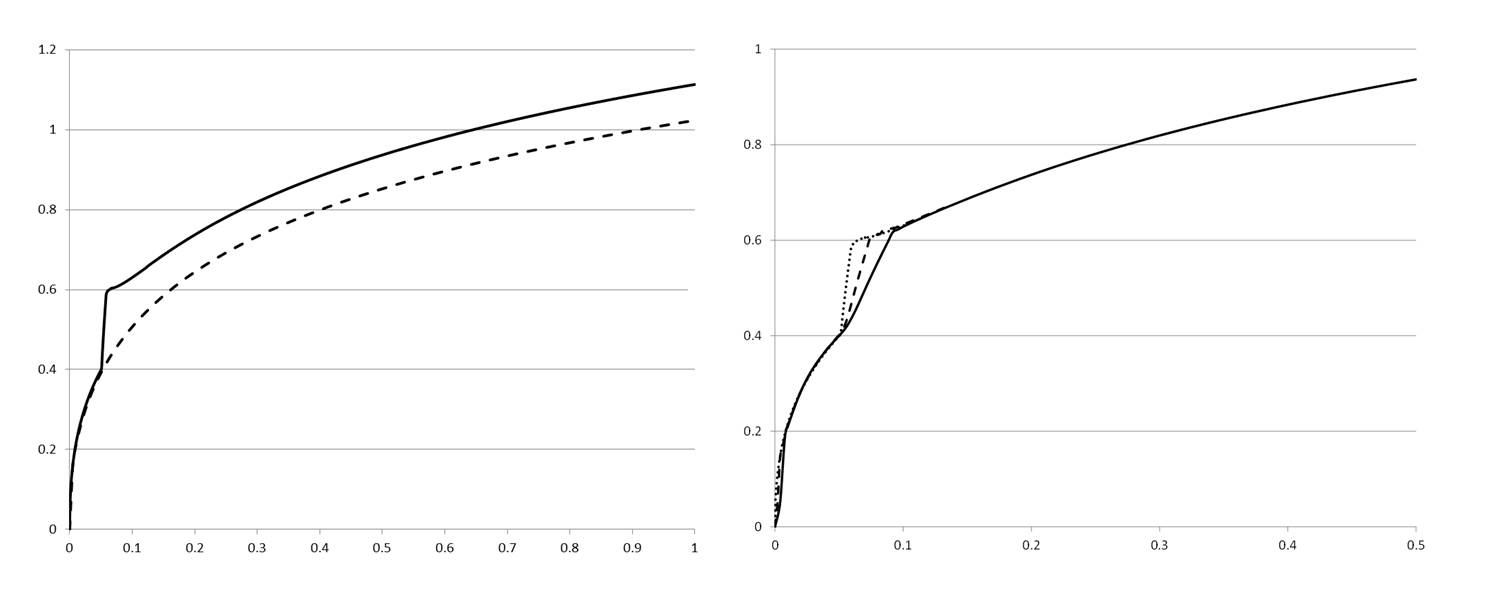}
\end{center}
\caption{\small{Boundary of Rost's barriers for uniform-like distributions.
\textbf{(Left)} $F'_\mu(x)=0.5\,\,\mathds{1}_{[0,2]}(x)$, $h=5\times 10^{-3}$ (dashed line) and $F'_\mu(x)=0.5\,\,\mathds{1}_{[0,0.4]}(x)+0.5\,\,\mathds{1}_{[0.6,2.2]}(x)$, $h=10^{-4}$ (solid line).
\textbf{(Right)} $F'_\mu(x)=0.5\,\,\mathds{1}_{[0,0.4]}(x)+0.5\,\,\mathds{1}_{[0.6,2.2]}(x)$, $h=4\times 10^{-3}$ (solid line), $h=2\times 10^{-3}$ (dashed line), $h=2\times 10^{-4}$ (fine dashed line); as expected the sharpness of the jump increases as $h\downarrow 0$.}}
\end{figure}

\begin{figure}[!ht]
\label{fig:3}
\begin{center}
\includegraphics[scale=0.6]{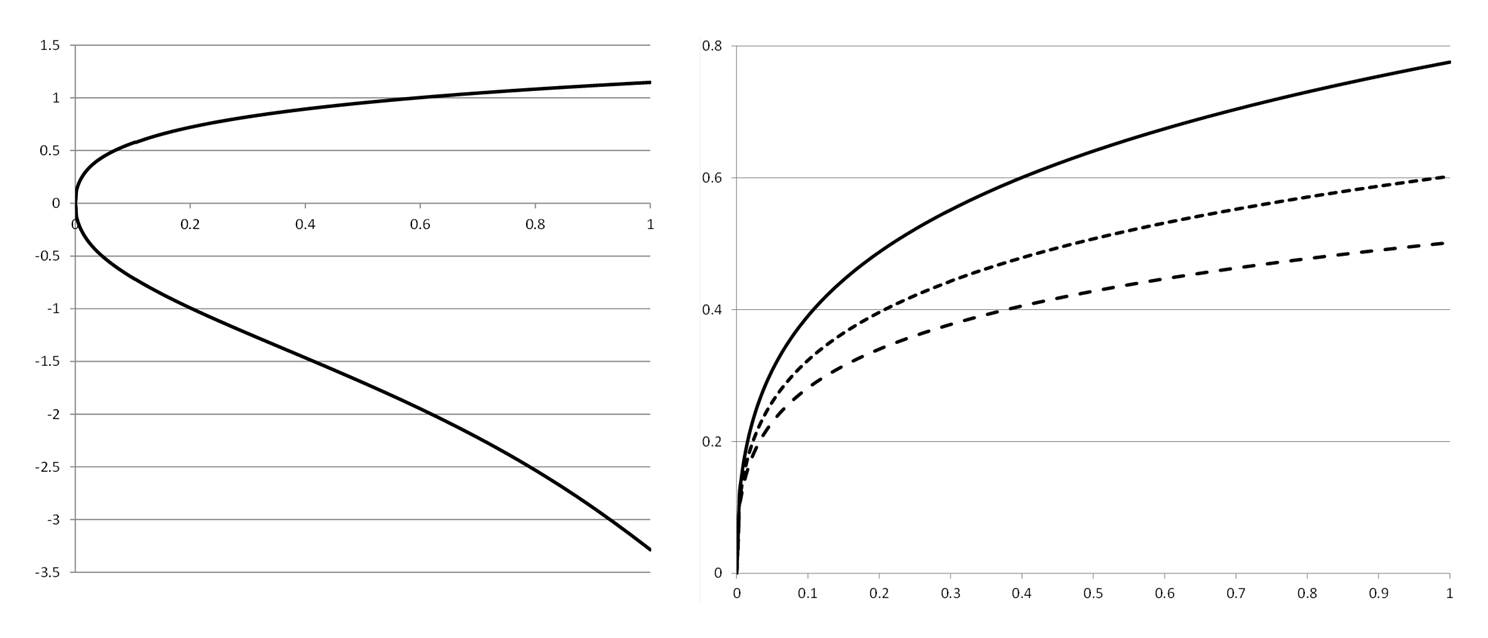}
\end{center}
\caption{\small{Boundary of Rost's barriers: \textbf{(Left)} $\mu=\mathcal{N}(m,\sigma^2)$ with mean $m=1$ and variance $\sigma^2=1$ ($h=10^{-3}$). \textbf{(Right)} $\mu=Exp(\lambda)$, $h=2\times10^{-3}$, $\lambda=1.5$ (solid line), $\lambda=2.5$ (fine dashed line), $\lambda=3.5$ (dashed line).}}
\end{figure}

\section{Numerical solution to the integral equations}\label{sec:num}

We proved above that $s_+$ and $s_-$ are the unique solutions to the equations \eqref{eq:intEq}. These equations cannot be solved analytically but in this section we describe a relatively simple algorithm for their numerical solution. 

The convergence of this algorithm has been studied in the literature in several specific examples. For instance in \cite{KK03} the authors deal with optimal exercise boundaries for American Put and Call options on dividend-paying assets, whereas in Section 5 of \cite{ZS09} the algorithm is applied to inverse first passage time problems. A broader overview of numerical methods for Volterra equations may be found in \cite{Linz} as well as in the list references of \cite{KK03} and \cite{ZS09}. 
\vspace{+5pt}

Now we illustrate the method. As in the previous section, we focus on time reversed boundaries $b_{\pm}(t)=s_{\pm}(T- t)$ for $t\in[0,T]$.
For $l=0,1,...,N$ we set $t_l=lh$ with $h=T/N$. Then for $k= 0,1,...,N-1$ we consider the following approximations of the left-hand side of the integral equations \eqref{eq:intEq}: 
\begin{align*}
&\int_{t_k}^T\hspace{-5pt}\int_{-b_-(u)}^{b_+(u)}\hspace{-3pt}p(t_k,\pm b_\pm(t),u,y)(\nu-\mu)(dy)\,du\hspace{-1pt}\approx\hspace{-1pt} h\hspace{-1pt}\sum_{l=k+1}^{N}\hspace{-3pt} \int_{-b_- (t_{l})}^{b_+ (t_{l})}\hspace{-3pt} p(t_k,b_\pm(t_k),t_{l},y)\hspace{-1pt}(\nu-\mu)(dy).
\end{align*}
Setting the above equal to zero, with a slight abuse of notation we obtain the approximating integral equations 
\begin{align}\label{Num-1} \hs{5pc}
&h\sum_{l=k+1}^{N} \int_{-b_- (t_{l})}^{b_+ (t_{l})} p\big(t_k,b_+(t_k),t_{l},y\big)(\nu-\mu)(dy)=0,\\
\label{Num-2}
&h\sum_{l=k+1}^{N} \int_{-b_- (t_{l})}^{b_+ (t_{l})} p\big(t_k,-b_-(t_k),t_{l},y\big)(\nu-\mu)(dy)=0,
\end{align}
which we are now going to solve recursively starting from $t_N=T$.

For $k=N-1$ and $b_{\pm}(t_N)=\hat{b}_{\pm}$ equations \eqref{Num-1} and \eqref{Num-2} become algebraic equations for $b_{+}(t_{N- 1})$ and $b_- (t_{N- 1})$, respectively. These can be solved  numerically using, e.g.,~Newton-Raphson or bisection method. Next we use the known values of $b_{+}(t_{N})$, $b_{+}(t_{N- 1})$, $b_- (t_{N})$ and $b_- (t_{N- 1})$ to obtain algebraic equations from \eqref{Num-1} and \eqref{Num-2} when $k=N-2$. By solving the latter numerically we get $b_{+}(t_{N- 2})$ and $b_- (t_{N- 2})$. The procedure is then iterated to compute $b_+(t_{N}),b_+(t_{N-1}),...,b_+(t_1),b_+(t_0)$ and
$b_-(t_{N}),b_-(t_{N-1}),...,b_-(t_1),b_-(t_0)$ as approximations of the optimal boundaries $b_+$ and $b_-$, respectively, at the
points $T,T-h,...,h,0$. 

Letting $h\to 0$, arguments similar to those developed in \cite{KK03,ZS09} guarantee convergence of the approximations to the actual boundaries. A complete proof of the convergence is outside the scopes of this work, however we observe it very clearly in our numerical experiments illustrated in Figures 2 and 3. 

\vspace{+15pt}


\noindent\ackn{The first named author was funded by EPSRC grant EP/K00557X/1}


\end{document}